\documentclass[12pt,a4paper]{article}
\usepackage[utf8]{inputenc}
\usepackage[T1]{fontenc}
\usepackage{amsmath}
\usepackage{amssymb}
\usepackage{amsthm}
\usepackage{graphicx}
\usepackage{mathrsfs}

\newtheorem{remark}{Remark}[section]
\newtheorem{theorem}{Theorem}[section]
\newtheorem{lemma}{Lemma}[section]
\newtheorem{definition}{Definition}[section]
\title{ Associating hypergraphs defined on loops}
\author {{Siddharth Malviy and Vipul Kakkar}\\ 
\vspace{0.1cm} Central University of Rajasthan \\Ajmer, India \\
Email:  malviysiddharth@gmail.com, vplkakkar@gmail.com }

\date{}
\begin{document}
\maketitle
\noindent \textbf{{Abstract.}}
	In this paper, we define a new hypergraph $\mathcal{H(V,E)}$ on a loop $L$, where $\mathcal{V}$ is the set of points of the loop $L$ and $\mathcal{E}$ is the set of hyperedges $e=\{x,y,z\}$ such that $x,y$ and $z$ associate in the order they are written. We call this hypergraph as the associating hypergraph on a loop $L$. We  study certain properites of associating hypergraphs on the Moufang loop $M(D_n,2)$, where $D_n$ denotes the dihedral group of order $2n.$\\
	


	\section{Introduction}

 Graphs are used in various domains. A lot of development and research on graphs have been done during the first half of twentieth century, but it is still an active field of research and applications. Hypergraphs is a generalization of the graphs. For more information one can see \cite{berge}. In \cite{bela}, extremal hypergraphs problems have been discussed.  In \cite{hamiltonian chain} and \cite{loose hamilton}, the condition for a hamiltonian chain in hypergraphs; in \cite{chromatic polynomial}, chromatic polynomial of hypergraphs; in \cite{hardness}, the hardness of $3$-uniform hypergraph coloring; in \cite{partioning}, partition of $3$-uniform hypergraphs; in \cite{eulerian}, eulerain properties of hyperegraphs; in \cite{linear cycle}, linear cycles in $3$-uniform hypergraphs; in \cite{matching polynomial}, matching polynomial of hypergraphs; in \cite{spectral properties}, spectral properties of uniform hypergraphs; in \cite{prime}, modules and prime of 3-uniform hypergraphs and in \cite{genus}, genus of complete $3$-uniform hypergraphs have been studied.

\noindent In $2024$, P J Cameron et al. defined the hypergraphs for algebraic structures, especially groups \cite{groups}. There has not been so much work done in hypergraphs defined on algebraic structures. In this paper, we define a new hypergraph on a loop. This is an initiative of a study of hypergraphs on non-associative structures.

\noindent A hypergraph $H=(V,E)$ is an ordered pair, where $V$ is a set of vertices and $E$ is a family of subsets of $V$. A hypergraph is $k$-uniform if every hyperedge include exactly $k$-vertices. A magma $(\mathrm{Q}, \cdot)$ is a non-empty set $\mathrm{Q}$ with a binary operation $\cdot$ on
$\mathrm{Q}$. A magma $(\mathrm{Q}, \cdot)$ is a quasigroup if for any $x, z \in \mathrm{Q}$ there is a unique $y$ such
that $x \cdot y = z$, and for any $y, z \in \mathrm{Q}$ there is a unique $x$ such that $x \cdot y = z$.  A quasigroup $(\mathrm{Q}, \cdot)$ is a loop if there is a neutral element $1 \in \mathrm{Q}$ such
that $1 \cdot x = x \cdot 1 = x$ for all $x \in \mathrm{Q}$. Elements in a loop need not satisfy the associative law.
 
\noindent Although graphs are limited to supporting pairwise interactions, hypergraphs are able to maintain multi-adic links, making them an ideal tool for modeling collabration networks and other scenarios. They have so many appplications in wide range of field. Hypergraph theory is used in psychology, genetics, to solve optimization problems and so on. Cellular mobile communication networks can be modeled using a hypergraph theory. In chemistry, the concept of molecular hyperegraph is generalization of molecular graph (see\cite{chemistry}). In biology, hypergraph theory is also used (see \cite{biology}). Hypergraph spectral clustering, which extends spectral graph theory with hypergraph Laplacian, is an example of a representative hypergraph learning technique (see \cite{spectral hypergraph}). The theory of hypergraphs  is given in data mining (see \cite{data mining}), image procesing ( see \cite{image}), matching learning (see \cite{machine learning}) and so on ( see \cite{machine}, \cite{geometry}).

 \noindent Throughout the article, $\mathbb{Z}_n$ will denote the cyclic group of order $n$. A hypergrpah $\mathcal{H(V,E)}$ will denote the associating hypergraph, where $\mathcal{V}$ is the set of points of loop and $\mathcal{E}$ is the set of hyperedges. A hyperedge consist of three elements when three elements say $x,y$ and $z$ will associative in the order they are written.

\noindent This article is structured as follows. In section $2$, we will study some preliminaries related to hypergraphs. In section $3$ and $4$, we will study associating hypergraphs for Moufang loop $M(D_n,2)$, where $D_n$ is the dihedral group of order $2n$.

 \section{Preliminaries}

 In this section, we will study some basic terms and definition that we have used in this paper. Some of them is a generalization of the definitions used in graph theory.







\begin{definition}
    \cite{bretto}  An edge $e$ is an empty edge if $e= \phi$, and a vertex $v$ is isolated vertex if $E(v)=\phi$.
\end{definition}

\begin{definition}
  \cite{bretto}  A hypergraph coloring is an assignment of colors to the vertices such that no edge has all vertices of the same color.
\end{definition}

\begin{definition}
   \cite{bretto} The chromatic number $\psi(H)$ of a hypergraph $H$ is the minimum number of distinct colors needed to color $H.$
\end{definition}

\begin{definition}
  \cite{bretto}  A strong coloring is a coloring in which vertices that share same hyperedge are assigned different colors.
\end{definition}

\begin{definition}
  \cite{bretto} The strong chromatic number $\overline{\psi}(H)$ of a hypergraph $H$ is the minimum number of distinct colors needed to strongly color $H.$
\end{definition}

\begin{definition}
  \cite{bretto}  A matching is a set of hyperedges such that no two hyperedges in the matching share a common vertex. 
\end{definition}

\begin{definition}
  \cite{bretto}  The matching number $\upsilon(H)$ of a hypergraph $H$ is the size of the largest matching in the hypergraph.
\end{definition}

\begin{definition}
   \cite{bretto} The covering number $\rho(H)$ of a hypergraph $H$ is the minimum number of hyperedges, that is required to ensure that every vertex is included atleast once. 
\end{definition}


\begin{definition}
    \cite{bretto} The independece set is a set of vertices that does not contain an edge as a subset. The independece number $\alpha(H)$ of a hypergraph $H$ is the size of largest independence set.
\end{definition}

\begin{definition}
    \cite{bretto}  A set of vertices, $T$ is a transversal if for every edge in hypergraph, there is a vertex in $T$ that is incident to that edge. The transversal number $\tau(H)$ of the hypergraph $H$ is the size of smallest transversal set.
\end{definition}

\begin{lemma} \label{trans property}
   \cite{formula} For the hypergraph $H(V,E)$ without empty edges, $\alpha(H)+\tau(H)= \mid V \mid.$
 \end{lemma}

\begin{definition} \label{Dv matrix}
   \cite{Dv De} For a hypergraph $H$, where $V$ is the set of vertices and E is the set of hyperedges, $D_v=(D_v(i,j))$ is a vertex-degree matrix with entries defined as
    $$D_v(i,j)= \begin{cases}
        deg(v_i) & i=j\\
        0 & i \neq j
    \end{cases},$$ where $deg(v_i)$ is the number of hyperedges containing vertex $v_i.$ The order of the $D_v$
matrix is $\mid V \mid  \times \mid V \mid.$
\end{definition}

\begin{definition} \label{De matrix}
    \cite{Dv De} For a hypergraph $H$, where $V$ is the set of vertices and E is the set of hyperedges, the hyperedge-degree matrix $D_e$ is a diagonal matrix, where each diagonal entry shows the size of a hyperedge. The order of the $D_e$
matrix is $\mid E \mid  \times \mid E \mid.$
\end{definition}

\begin{definition}
    \cite{JAR} The length of the shortest path that connects two vertices $u$ and $v$ in a hypergraph represents the distance $d(u,v)$ between them.
\end{definition}





\section{Associating hypergraphs on loops}

Let $L$ be a loop. Let   $x,y,z \in L$. If $$(x \cdot y) \cdot z = x \cdot ( y \cdot z),$$
then we say that $x,y$ and $z$ associate in the order they are written. It may happen that the three elements $x,y$ and $z$ associate in an order but they do not associate in other order.

\noindent For example, consider an example of the loop of order $5$ (see \cite[p.22, 1.78]{quasigroup book})\\

 \begin{center}
\begin{tabular}{c|ccccc}
     * & 0 & 1 & 2 & 3 & 4  \\ \hline
     0 & 0 & 1 & 2 & 3 & 4  \\
     1 & 1 & 0 & 3 & 4 & 2 \\
     2 & 2 & 4 & 0 & 1 & 3\\
     3 & 3 & 2 & 4 & 0 & 1\\
     4 & 4 & 3 & 1 & 2 & 0
 \end{tabular}
 \end{center}


\noindent Here, 
$$ (1* 2) * 3=  1*( 2 * 3),$$

\noindent  but 
$$ (1 * 3) * 2 \neq 1 *( 3 * 2).$$

This prompt us to define a hypergraph $\mathcal{H(V,E)}$ as follows:

\noindent Let the set of vertices $\mathcal{V}$ is the set of points of a loop $L$. The three distinct elements $x,y,z \in \mathcal{V}$ form an hyperedge if $x \cdot(y \cdot z) = (x \cdot y) \cdot z.$ We call this hypergraph as {associating hypergraph}. One can observe that the associating hypergraph on a loop is a $3$-uniform directed hypergraph.\\




\noindent A loop $(M, \cdot)$ is a Moufang loop if it satisfies any of the three (equivalent) Moufang identities\cite[2.1, p.1]{Moufang loop}:

\begin{enumerate} 
    \item $(x \cdot y) \cdot (z \cdot x) = (x \cdot (y \cdot z)) \cdot x$
    \item  $x \cdot (y \cdot (z \cdot y)) = ((x \cdot y) \cdot z) \cdot y$
    \item $x \cdot (y \cdot (x \cdot z)) = ((x \cdot y) \cdot x) \cdot z$
\end{enumerate}

\noindent Let $ (G, \cdot)$ be a non-abelian group and $\mathbb{Z}_2$ be the cyclic group of order $2$. Then we consider a set $M = \{(g, \alpha) | g \in G, \alpha \in \mathbb{Z}_2\}$. Define $\circ$ on $M$ as 
\begin{equation}\label{equation-1}
    (g_1, \alpha_1)\circ(g_2, \alpha_2) =(g_1 ^{1-\alpha_2}\cdot g_2 ^{
(-1)^{\alpha_1}} \cdot g_1 ^{\alpha_2}, \alpha_1 + \alpha_2).
\end{equation}
Then $(M, \circ)$ is a Moufang loop. We denote this Moufang loop by $M(G,2).$  Moufang loop $M(G,2)$ is not associative if and only if $G$ is not abelian group. For more information one can see \cite{Moufang loop}.
We will study the associating hypergraph for Moufang loop $M(D_n,2)$ on dihedral group $D_n(n \geq 3),$ where \[D_n= \langle x,y \mid x^n=y^2=1, xy=yx^{-1} \rangle.\]

\noindent Let $(g_1, \alpha_1), (g_2, \alpha_2), (g_3, \alpha_3) \in M(D_n,2).$ To find which three elements of $M(G,2)$ associates, we have to consider the following eight cases
\begin{align*}
    (\alpha_1, \alpha_2, \alpha_3) \in &\{(0,0,0),(0,1,0),(0,0,1),(0,1,1),\\& (1,0,1),(1,1,0),(1,0,0),(1,1,1)\}.
\end{align*}

 \noindent Let $Z(D_n)$ denote the center of dihedral group $D_n.$ Let $R=\{x,x^2,\hdots, x^n\}\setminus Z(D_n)$ and $S=\{y,xy,x^2y,\hdots,x^{n-1}y\}.$ Let $^nC_r$ denote $\frac{n!}{ r! \times (n-r)!}$, where $n,r \in \mathbb{N}.$\\



\noindent\textbf{Case 1.}  Let $(\alpha_1, \alpha_2, \alpha_3)=(0,0,0).$ Then all three elements are from the group $G$. Therefore, the total number of hyperedges from this case is  $6\times(^{2n}C_3).$\\

\noindent\textbf{Case 2.} Let $(\alpha_1, \alpha_2, \alpha_3)=(0,0,1).$ Since $(g_1,0) \neq (g_2,0)$, $g_1 \neq g_2$. \vspace{0.1cm}\\
Note that $$((g_1, 0)\circ (g_2, 0)) \circ (g_3,1)= (g_3 g_1 g_2,1)$$ and $$(g_1, 0)\circ ((g_2, 0) \circ (g_3,1))= (g_3 g_2 g_1,1).$$

\noindent Therefore, $(g_1,0), (g_2,0)$ and $(g_3,1)$ associate if and only if $$g_1g_2= g_2 g_1.$$

\noindent  First, assume that $n$ is odd. The elements $g_1$ and $g_2$ commute if and only if either $g_1, g_2 \in R$ or any one of $g_1$ or $g_2$ is in $Z(D_n).$

\noindent Therefore, the total number of hyperedges from this subcase is $$ ^{n-1}C_1 \times ^{n-2}C_1 \times ^{2n}C_1 + 2 \cdot(^{1}C_1 \times ^{2n-1}C_1 \times ^{2n}C_1)$$ $$=[2n(n-1)(n-2)] + [4n(2n-1)].$$

\noindent  Now, assume $n$ is even.  The elements $g_1$ and $g_2$ commute if and only if one of the following holds

\begin{itemize}
    \item $g_1, g_2 \in R.$
    \item Either of  $g_1$ or $g_2$ is in $Z(D_n)$ but not both.
 \item Both  of  $g_1$ and  $g_2$  are in $Z(D_n)$.
    \item  Both $g_1, g_2 \in S$ such that $g_1 g_2= g_2 g_1.$
\end{itemize}

\noindent Therefore, the total number of hyperedges from this subcase is $$ ^{n-2}C_1 \times ^{n-3}C_1 \times ^{2n}C_1 + 4 \cdot(^{1}C_1 \times ^{2n-1}C_1 \times ^{2n}C_1) + 2 \times ^{2n}C_1 + n \times ^{2n}C_1$$ $$=[2n(n-2)(n-3)]+ [8n(2n-1)]+4n+ 2n^2.$$

\noindent\textbf{Case 3.} Let $(\alpha_1, \alpha_2, \alpha_3)=(0,1,0).$ Since $(g_1,0) \neq (g_3,0)$, $g_1 \neq g_3$. \vspace{0.1cm}\\
Note that $$((g_1, 0)\circ (g_2, 1)) \circ (g_3,0)= (g_2 g_1 g_3 ^{-1},1)$$ and $$(g_1, 0)\circ ((g_2, 1) \circ (g_3,0))= (g_2 g_2 ^{-1} g_1,1).$$

\noindent Therefore, $(g_1,0), (g_2,1)$ and $(g_3,0)$ associate if and only if $$g_1g_3^{-1}= g_3^{-1} g_1 \Leftrightarrow g_1g_3= g_3g_1.$$

\noindent By the similar arguments as in Case 2, the total number of hyperedges is 
$$ ^{n-1}C_1 \times ^{n-2}C_1 \times ^{2n}C_1 + 2 \cdot(^{1}C_1 \times ^{2n-1}C_1 \times ^{2n}C_1)$$ $$=[2n(n-1)(n-2)] + [4n(2n-1)],$$ when $n$ is odd.
$$ ^{n-2}C_1 \times ^{n-3}C_1 \times ^{2n}C_1 + 4 \cdot(^{1}C_1 \times ^{2n-1}C_1 \times ^{2n}C_1) + 2 \times ^{2n}C_1 + n \times ^{2n}C_1$$ $$=[2n(n-2)(n-3)]+ [8n(2n-1)]+4n+ 2n^2,$$ when $n$ is even.\\






\noindent\textbf{Case 4.} Let $(\alpha_1, \alpha_2, \alpha_3)=(1,0,0).$ Since $ (g_2,0) \neq (g_3,0)$, $g_2 \neq g_3$. \vspace{0.1cm}\\
Note that $$((g_1, 1)\circ (g_2, 0)) \circ (g_3,0)= (g_1 g_2^{-1} g_3^{-1},1)$$ and  $$(g_1, 1)\circ ((g_2, 0) \circ (g_3,0))= (g_1 g_3 ^{-1} g_2^{-1},1).$$

\noindent Therefore, $(g_1,1), (g_2,0)$ and $(g_3,0)$ associate if and only if $$g_2^{-1}g_3^{-1}= g_3^{-1} g_2^{-1} \Leftrightarrow g_2g_3= g_3g_2.$$

\noindent By the similar arguments as in Case 2, the total number of hyperedges is 
$$ ^{n-1}C_1 \times ^{n-2}C_1 \times ^{2n}C_1 + 2 \cdot(^{1}C_1 \times ^{2n-1}C_1 \times ^{2n}C_1)$$ $$=[2n(n-1)(n-2)] + [4n(2n-1)],$$ when $n$ is odd.
$$ ^{n-2}C_1 \times ^{n-3}C_1 \times ^{2n}C_1 + 4 \cdot(^{1}C_1 \times ^{2n-1}C_1 \times ^{2n}C_1) + 2 \times ^{2n}C_1 + n \times ^{2n}C_1$$ $$=[2n(n-2)(n-3)]+ [8n(2n-1)]+4n+ 2n^2,$$ when $n$ is even.\\

\noindent\textbf{Case 5.} Let $(\alpha_1, \alpha_2, \alpha_3)=(0,1,1).$ Since $ (g_2,1) \neq (g_3,1)$, $g_2 \neq g_3$. \vspace{0.1cm}\\
Note that $$((g_1, 0)\circ (g_2, 1)) \circ (g_3,1)= ({g_3}^{-1} {g_2} g_1,0)$$ and  $$(g_1, 0)\circ ((g_2, 1) \circ (g_3,1))= ({g_1}^{-1} {g_3}^{-1} g_2 {g_1}^2,0).$$ So, in this case  $(g_1,0), (g_2,1)$ and $(g_3,1)$ associates if and only if $${g_3}^{-1} {g_2} g_1={g_1}^{-1} {g_3}^{-1} g_2 {g_1}^2 \Leftrightarrow g_1 (g_3^{-1} g_2)= (g_3^{-1} g_2) g_1.$$

\noindent Now, if we choose $g_1 \in Z(D_n),$ then $g_1(g_3^{-1} g_2)= (g_3^{-1} g_2)g_1$ for every $g_3^{-1} g_2 \in D_n.$
Therefore, we have $2n \times (2n-1)$ choices to form hyperedges when $n$ is odd and $2 \times 2n \times (2n-1)$ choices to form hyperedges when $n$ is even.\\

\noindent Now, we discuss the choices of hypergraph when $g_1 \in R.$ If we choose $g_1=x,$ then $xg_3^{-1} g_2= g_3^{-1} g_2x.$ Therefore, the choices for $g_3^{-1}g_2 \in \{x,x^2,\hdots,x^{n-1}\}.$ 

\noindent If we choose $g_3^{-1}g_2=x$ then $g_2=g_3 x$. Therefore, we have $2n$ choices for $g_3$ to form hyperedges.
Similar argument holds for other choices of $g_3^{-1}g_2.$ Therefore, for $g_1=x$, we have
$(n-1) \times 2n$ ways to form hyperedges.

\noindent Similarly, for any $g_1 \in R$ other than $x$, we have ($n-1) \times 2n$ ways to form hyperedges. Therefore, total number of hyperedges when $g_1 \in R$ is $(n-1) \times (n-1) \times 2n,$ when $n$ is odd and $(n-2) \times (n-1) \times 2n$, when $n$ is even.\\

\noindent Finally, we discuss the choices of hypergraph when $g_1 \in S.$ Let $g_1=y,$ then $yg_3^{-1} g_2= g_3^{-1} g_2y.$  Therefore, the choices for $g_3^{-1}g_2 \in \{y\}$ when $n$ is odd and $g_3^{-1}g_2 \in \{y, x^{\frac{n}{2}}, x^{\frac{n}{2}}y\}$ when $n$ is even. If we choose $g_3^{-1}g_2=y$ then $g_2=g_3 y$. Therefore, we have $2n$ choices for $g_3$ to form hyperedges. Similar argument holds for other choices ($x^{\frac{n}{2}}, x^{\frac{n}{2}}y$) of $g_3^{-1}g_2$, when $n$ is even. Therefore, total number of hyperedges when $g_1 \in S$ is $n \times 2n,$ when $n$ is odd. And total number of hyperedges when $g_1 \in S$ is $ n \times 3 \times 2n,$ when $n$ is even.\\

 \noindent Hence, total number of hyperedges from this case is:

\noindent  $[2n \times (2n-1)] + [(n-1 )\times (n-1 )\times 2n] + [n \times 2n],$ when $n$ is odd.
 
 \noindent $[2 \times 2n \times (2n-1)] +[( n-2)\times (n-1 )\times 2n] + [n \times 3 \times 2n],$  when $n$ is even.\\

 \noindent\textbf{Case 6.} Let $(\alpha_1, \alpha_2, \alpha_3)=(1,0,1).$ Since $(g_1,1) \neq (g_3,1)$, $g_1 \neq g_3$. \vspace{0.1cm}\\
Note that $$((g_1, 1)\circ (g_2, 0)) \circ (g_3,1)= ({g_3}^{-1} g_1 {g_2}^{-1}, 2)$$ and 
 $$(g_1, 0)\circ ((g_2, 1) \circ (g_3,1))= ({g_2}^{-1} {g_3}^{-1} g_1, 2).$$ So, in this case  $(g_1,1), (g_2,0)$ and $(g_3,1)$ associates if and only if $${g_3}^{-1} g_1 {g_2}^{-1}={g_2}^{-1} {g_3}^{-1} g_1 \Leftrightarrow g_2 (g_3 ^{-1}g_1)= (g_3^{-1} g_1) g_2.$$ 
 
  \noindent By the similar arguments as in Case 5, the total number of hyperedges from this case is:

\noindent  $[2n \times (2n-1)] + [(n-1 )\times (n-1 )\times 2n] + [n \times 2n],$ when $n$ is odd.
 
 \noindent $[2 \times 2n \times (2n-1)] +[( n-2)\times (n-1 )\times 2n] + [n \times 3 \times 2n],$  when $n$ is even.\\
 
 \noindent\textbf{Case 7.} Let $(\alpha_1, \alpha_2, \alpha_3)=(1,1,0).$ Since $(g_1,1) \neq (g_2,1)$, $g_1 \neq g_2$. \vspace{0.1cm}\\
Note that $$((g_1, 1)\circ (g_2, 1)) \circ (g_3,0)= ({g_2}^{-1} g_1 g_3, 2)$$ and
$$(g_1, 1)\circ ((g_2, 1) \circ (g_3,0))= (g_3 {g_2}^{-1} g_1, 2).$$ So, in this case  $(g_1,1), (g_2,1)$ and $(g_3,0)$ associates if and only if $$({g_2}^{-1} g_1) g_3=g_3( {g_2}^{-1} g_1).$$ By the similar arguments as in Case 5, the total number of hyperedges from this case is:

\noindent  $[2n \times (2n-1)] + [(n-1 )\times (n-1 )\times 2n] + [n \times 2n],$ when $n$ is odd.
 
 \noindent $[2 \times 2n \times (2n-1)] +[( n-2)\times (n-1 )\times 2n] + [n \times 3 \times 2n],$  when $n$ is even.\\

\noindent\textbf{Case 8.} Let $(\alpha_1, \alpha_2, \alpha_3)=(1,1,1).$ Since $(g_1,1), (g_2,1)$ and  $(g_3,1)$ are distinct,  $g_1, g_2$ and $g_3$ are distinct.\vspace{0.1cm}\\
Note that $$((g_1, 1)\circ (g_2, 1)) \circ (g_3,1)= (g_3 {g_2}^{-1} {g_1},1 )$$ and 
 $$(g_1, 1)\circ ((g_2, 1) \circ (g_3,1))= ({g_1}^{-1} {g_2}^{-1} g
_3 {g_1}^{2}, 1).$$ Therefore, $(g_1,1),(g_2,1)$ and $(g_3,1)$ associate if and only if $$g_3 {g_2}^{-1} {g^1}={g_1}^{-1} {g_2}^{-1} g
_3 {g_1}^{2}$$ 
\begin{equation}\label{equation}
    \Leftrightarrow g_2(g_1 g_3)= (g_3 g_1)g_2.
\end{equation}

\noindent First, let $g_2 \in Z(D_n)$. Then by equation (\ref{equation}), $g_1g_3=g_3g_1.$ This condition satisfied if and only if one of the following holds.
\begin{itemize}
    \item Both $g_1$ and $g_3$ is from set $R.$
    \item Both $g_1$ and $g_3$ is from set $S$ such that $g_1g_3=g_3g_1.$ Which is possible only when $n$ is even.
    \item Either $g_1 \in Z(D_n)$ or $g_3 \in Z(D_n).$ Since $g_1 \neq g_2$ and $g_2 \neq g_3,$ this is only possible when $n$ is even.
\end{itemize}

\noindent Now, let $g_2 \in R$. Then, we have four cases for the value of $g_1$ and $g_3$ for which it forms hyperedge. These cases are mentioned below:

\begin{itemize}
    \item Both $g_1$ and $g_3$ is from $R$ or $Z(D_n)$. Then  all three elements always satisfy equation (\ref{equation}).
    \item Both $g_1$ and $g_3$ is from $S$. Then by equation (\ref{equation}), the condition for forming hyperedges is $g_1g_3=g_3g_1.$ This condition holds only when $n$ is even.
    \item  $g_1$ is from set $R$, and $g_3$ is from set $S$. Then by equation  (\ref{equation}), the condition for forming hyperedges is $g_2g_1 \in Z(D_n).$
    \item  $g_1 \in S$ and $g_3 \in R.$ Suppose $g_1=x^jy, g_3=x^k$ and $g_2=x^i.$ Then by equation (\ref{equation}), the condition for forming hyperedges is $k={\frac{n}{2}+i}.$ This condition holds only when $n$ is even.
\end{itemize}

\noindent Now, let $g_2 \in S$. Then, we have five cases for the value of $g_1$ and $g_3$ for which it forms hyperedge. These cases are mentioned below:

\begin{itemize}
    \item Both $g_1$ and $g_3$ is from $R$, then they form hyperedges if and only if $g_1g_3 \in Z(D_n).$
    \item Both $g_1$ and $g_3$ is from $Z(D_n)$, then they will form hyperedges for any value of $g_2.$ This condition only holds when $n$ is even.
    \item Both $g_1$ and $g_3$ is from $S$, then they form hyperedges if and only if $g_1g_3 = x^{\frac{n}{2}}.$ This condition only holds when $n$ is even.
    \item  $g_1$ is from set $S$ and $g_3$ is from set $R$ or $Z(D_n)$. Then by  equation (\ref{equation}), the condition for forming hyperedges is $g_2g_1=g_1g_2.$ This condition only holds when $n$ is even.
    \item  $g_1 \in R$ or $R \in Z(D_n)$ and $g_3 \in S.$ Suppose $g_1=x^j, g_3= x^ky$ and $g_2= x^iy.$ Then by equation (\ref{equation}), the condition for forming hyperedges is $i-k=k-i.$ This condition only holds when $n$ is even.
\end{itemize}

\noindent Consider all the three cases ($g_2 \in Z(D_n), g_2 \in R, g_2 \in S)$ mentioned above, total number of hyperedges from this case is:
\begin{align*}
 1 \times [(n-1)(n-2)] + (n-1) \times [(n-1) (n-2) + 1(n)] + n \times [1(n-2)],
\end{align*} when $n$ is odd.
\begin{align*}
    & 2 \times [ (n-2) (n-3)+ 1(n) + 2(2n-2)] \\&+ (n-2) \times [ (n-1) (n-2)+ 1(n) + 2(n) + 1(n)]  \\&+ n \times [ 2 (n-2) + 2 + 1(n) + 1(n) + 1(n)],
\end{align*} when $n$ is even.\\
 

\noindent The order $\mid \mathcal{V} \mid$ of the  associating hypergraph $\mathcal{H(V,E)}$ is $4n$, and the size $\mid \mathcal{E} \mid$ of the associating hypergraph $\mathcal{H(V,E)}$ is the total number of hyperedges from Case $1$ to Case $8$.
 

\section{Some properties of associating hypergraph $\mathcal{H(V,E)}$ on Moufang loop $M(D_n,2)$}

In this section, we will study some properties related to associating hypergraphs $\mathcal{H(V,E)}$ on  Moufang loop $M(D_n,2)$, where $n \geq 3 $.
 
 \begin{theorem}
     The distance $d(u,v)$ between any two vertices $u$ and $v$ of the associating hypergraph $\mathcal{H(V,E)}$ is $1$.
 \end{theorem}

 \begin{proof}
      In $\mathcal{H(V,E)}$, any two vertices is a part of some hyperedges so the shortest distance between then is always $1$. Therefore, the distance $d(u,v)$ between any two vertices $u$ and $v$ of the associating hypergraph $\mathcal{H(V,E)}$ is $1$.
 \end{proof}

 \begin{theorem} \label{degree theorem}
      The degree of the elements of the associating hypergraph $\mathcal{H(V,E)}$ when $n$ is even is
      \[deg(v)= \begin{cases}
          A,  & v = (g,0), g \in Z(D_n)\\
          B,  & v = (g,0), g \in R\\
          C,  & v = (g,0), g \in S\\
          D,  & v = (g,1), g \in Z(D_n)\\
          E,  & v = (g,1), g \in R\\
          F,  & v = (g,1), g \in S\\
      \end{cases},\]
      \noindent where,
      
      \begin{align*}
          A=& [3 \times ^{2n-1} C_2] + 3\times[2(^{2n-1} C_1`\times ^{2n}C_1)+ 2 \times ^{2n}C_1]+ 3\times[{2n} \times (2n-1)],\\ \\
          B=&[3 \times ^{2n-1} C_2] + 3\times[2 (^{n-3} C_1) + 4 \times (^{2n}C_1)]+ 3\times[(n-1) \times 2n],\\ \\
          C=&[3 \times ^{2n-1} C_2] + 3\times[4 \times (^{2n}C_1) + ^{2n}C_1]+ 3\times[3 \times 2n],
      \end{align*}
      \begin{align*}
       D=&3\times[^{n-2} C_1\times ^{n-3}C_1 + 4(^{2n-1}C_1) + 2 + n] \\+& 3\times[2 \times (2(2n-1)) + (n-2) \times(2 \times (n-1)) + n \times (3 \times 2)]\\+& (n-2) (n-3)+ 1(n) + 1(2n-2) + (n-2) \times 2(n-2) + n \times (2+ 1 + 1),
   \end{align*}
    \begin{align*}
       E=&3\times[^{n-2} C_1\times ^{n-3}C_1 + 4(^{2n-1}C_1) + 2 + n] \\+& 3\times[2 \times (2(2n-1)) + (n-2) \times(2 \times (n-1)) + n \times (3 \times 2)]\\+& 2 \times (2(n-3)) + 2 \times (2(2n-2) + (n-1)(n-2) + 1(n) + 2(n) + 1(n) \\+& n \times (2(2) + 1 + 1),
   \end{align*}
    \begin{align*}
       F=&3\times[^{n-2} C_1 \times ^{n-3}C_1 + 4(^{2n-1}C_1) + 2 + n] \\+& 3\times[2 \times (2(2n-1)) + (n-2) \times(2 \times (n-1)) + n \times (3 \times 2)]\\+& 2(2+2) +(n-2)\times (1+ 2+ 1)+ 1\times(2(n-2)+2+1(n)+1(n)+ 1(n)).
   \end{align*}
   \end{theorem}

  \begin{proof}
       We will calculate the degree of all different types of elements separately.\\

   \noindent  Consider the elements of type $(g,0)$ such that $g \in Z(D_n).$ Then from Case $1$, the number of hyperedges incident to $(g,0),$ where $ g \in Z(D_n)$ is $3 \times ^{2n-1} C_2.$

   \noindent  From each Case $2,3$ and $4$, the number of hyperedges incident to $(g,0),$ where $g \in Z(D_n)$ is $2(^{2n-1} C_1`\times ^{2n}C_1)+ 2 \times ^{2n}C_1.$

   \noindent From each Case $5,6$ and $7$, the number of hyperedges incident to $(g,0),$ where $g \in Z(D_n)$ is $2n \times (2n-1).$

   \noindent Therefore, the degree of the elements of type $(g,0), g \in Z(D_n)$ is \[[3 \times ^{2n-1} C_2] + 3\times[2(^{2n-1} C_1`\times ^{2n}C_1)+ 2 \times ^{2n}C_1]+ 3\times[{2n} \times (2n-1)]= A.\]

   \noindent  Consider the elements of type $(g,0)$ such that $g \in R.$ Then from Case $1$, the number of hyperedges incident to $(g,0),$ where $g \in R $ is $3 \times ^{2n-1} C_2.$

   \noindent From each Case $2,3$ and $4$, the number of hyperedges incident to $(g,0),$ where $g \in  R$ is $2 (^{n-3} C_1) + 4 \times (^{2n}C_1).$

   \noindent From each Case $5,6$ and $7$, the number of hyperedges incident to $(g,0),$ where $g \in R$ is $(n-1) \times 2n.$

   \noindent Therefore, the degree of the elements of type $(g,0), g \in R$ is \[[3 \times ^{2n-1} C_2] + 3\times[2 (^{n-3} C_1) + 4 \times (^{2n}C_1)]+ 3\times[(n-1) \times 2n]=B.\]

   \noindent  Consider the elements of type $(g,0)$ such that $g \in S.$ Then from Case $1$, the number of hyperedges incident to $(g,0),$ where  $g \in S$ is $3 \times ^{2n-1} C_2.$

   \noindent From each Case $2,3$ and $4$, the number of hyperedges incident to $(g,0),$ where $g \in S$ is $ 4 \times (^{2n}C_1) + ^{2n}C_1.$

   \noindent From each Case $5,6$ and $7$, the number of hyperedges incident to $(g,0),$ where $g \in S$ is $ 3 \times 2n.$

   \noindent Therefore, the degree of the elements of type $(g,0), g \in S$ is \[[3 \times ^{2n-1} C_2] + 3\times[4 \times (^{2n}C_1) + ^{2n}C_1]+ 3\times[3 \times 2n]=C.\]

   \noindent  Consider the elements of type $(g,1)$ such that $g \in Z(D_n).$ Then from each Case $2,3$ and $4$, the number of hyperedges incident to $(g,1),$ where $ g \in Z(D_n)$ is $^{n-2} C_1`\times ^{n-3}C_1 + 4(^{2n-1}C_1) + 2 + n.$

   \noindent From each Case $5,6$ and $7$, the number of hyperedges incident to $(g,1),$ where $g \in Z(D_n)$ is $ 2 \times (2(2n-1)) + (n-2) \times(2 \times (n-1)) + n \times (3 \times 2).$

   \noindent From Case $8$, the number of hyperedges incident to $(g,1),$  where $g \in Z(D_n)$ is $(n-2) (n-3)+ 2(n) + 1(2n-2) + (n-2) \times (2(n-2)) + n \times (2+ 1 + 1).$

   \noindent Therefore, the degree of the elements of type $(g,1), g \in Z(D_n)$ is 
   \begin{align*}
       &3\times[^{n-2} C_1`\times ^{n-3}C_1 + 4(^{2n-1}C_1) + 2 + n] \\+& 3\times[2 \times (2(2n-1)) + (n-2) \times(2 \times (n-1)) + n \times (3 \times 2)]\\+& (n-2) (n-3)+ 2(n) + 1(2n-2) + (n-2) \times 2(n-2) + n \times (2+ 1 + 1) \\=& D.
   \end{align*}

    \noindent  Consider the elements of type $(g,1)$ such that $g \in R.$  Then from  each Case $2,3$ and $4$, the number of hyperedges incident to $(g,1),$ where $g \in R$ is $^{n-2} C_1`\times ^{n-3}C_1 + 4(^{2n-1}C_1) + 2 + n.$

   \noindent From each Case $5,6$ and $7$, the number of hyperedges incident to $(g,1),$ where $ g \in R$ is $ 2 \times (2(2n-1)) + (n-2) \times(2 \times (n-1)) + n \times (3 \times 2).$

   \noindent From Case $8$, the number of hyperedges incident to $(g,1),$ where $ g \in R$ is $ 2 \times (2(n-3)) + 2 \times (2(2n-2) + (n-1)(n-2)+(n-3)(2(n-2)) + 1(n) + 2(n) + 2(n) + n \times (2(2) + 1 + 1).$

   \noindent Therefore, the degree of the elements of type $(g,1), g \in R$ is 
   \begin{align*}
       &3\times[^{n-2} C_1`\times ^{n-3}C_1 + 4(^{2n-1}C_1) + 2 + n] \\+& 3\times[2 \times (2(2n-1)) + (n-2) \times(2 \times (n-1)) + n \times (3 \times 2)]\\+& 2 \times (2(n-3)) + 2 \times (2(2n-2) + (n-1)(n-2)+(n-3)(2(n-2))\\+& 1(n) + 2(n) + 2(n) + n \times (2(2) + 1 + 1) \\=& E.
   \end{align*}

    \noindent  Consider the elements of type $(g,1)$ such that $g \in S,$ Then from each Case $2,3$ and $4$, number of hyperedges incident to $(g,1),$ where $ g \in S$ is $^{n-2} C_1`\times ^{n-3}C_1 + 4(^{2n-1}C_1) + 2 + n.$

   \noindent From each Case $5,6$ and $7$, number of hyperedges incident to $(g,1),$ where $ g \in S$ is $ 2 \times (2(2n-1)) + (n-2) \times(2 \times (n-1)) + n \times (3 \times 2).$

   \noindent From Case $8$, number of hyperedges incident to $(g,1),$ where $ g \in S$ is $ 2(2+2) +(n-2)\times (1+ 2+ 1)+ 1 \times ((n-2)(n-3))+ 1 \times 2 + 1 \times n + 2 (n-1) + 2(n) + 2(n).$

   \noindent Therefore, the degree of the elements of type $(g,1), g \in S$ is 
   \begin{align*}
       &3\times[^{n-2} C_1`\times ^{n-3}C_1 + 4(^{2n-1}C_1) + 2 + n] \\+& 3\times[2 \times (2(2n-1)) + (n-2) \times(2 \times (n-1)) + n \times (3 \times 2)]\\+& 2(2+2) +(n-2)\times (1+ 2+ 1)+ 1 \times ((n-2)(n-3))+ 1 \times 2 \\+& 1 \times n + 2 (n-1) + 2(n) + 2(n) \\=& F.
   \end{align*}
  \end{proof}

  \begin{theorem} \label{degree theorem odd}
      The degree of the elements of the associating hypergraph $\mathcal{H(V,E)}$ when $n$ is odd is
      \[deg(v)= \begin{cases}
          A,  & v = (g,0), g \in Z(D_n)\\
          B,  & v = (g,0), g \in R\\
          C,  & v = (g,0), g \in S\\
          D,  & v = (g,1), g \in Z(D_n)\\
          E,  & v = (g,1), g \in R\\
          F,  & v = (g,1), g \in S\\
      \end{cases},\]
      \noindent where,
      
      \begin{align*}
          A=& 3(2n-1) + 3[ 2(2n-1)(2n)]+3[2n(2n-1)],\\ 
          B=&  3(2n-1) + 3[2(n-2)+2(2n)]+3[(n-1)(2n)],\\ 
          C=&  3(2n-1) + 3[2(2n)] + 3[2n],
      \end{align*}
      \begin{align*}
       D=& 3[(n-1)(n-2)+2(2n-1)]+ 3[2(2n-1)+2(n-1)(n-1)+ 2(n)] \\+& (n-1)(n-2)+2(n-1)(n-2), 
   \end{align*}
    \begin{align*}
       E=& 3[(n-1)(n-2)+2(2n-1)]+3[2(2n-1)+ 2(n-1)(n-1)+2(n)]\\+& 2(n-2)+ (n-2)((n-1)+2) +n,
   \end{align*}
    \begin{align*}
       F=& 3[(n-1)(n-2)+2(2n-1)]+3[2(2n-1)+2(n-1)(n-1)+2(n)]\\+& 1(n-1)+1(n-1).
   \end{align*}
   \end{theorem}

   \begin{proof}
       The proof is similar to the proof of Theorem \ref{degree theorem}.
   \end{proof}

  \noindent  From the Definition \ref{Dv matrix}, the degree-vertex matrix $D_v$ of the associating hypergraph  $\mathcal{H(V,E)}$
    \begin{equation*}
        D_v=\mathrm{diag}( A, \underbrace{B}_{(n-2)- times}, \underbrace{C}_{n-times}, D, \underbrace{E}_{(n-2)-times}, \underbrace{F}_{n-times} ),
    \end{equation*} when $n$ is even and $A, B, C, D, E$ and $F$ are given in Theorem \ref{degree theorem}.\\

   \noindent Similarly, the degree-vertex matrix $D_v$ of the associating hypergraph  $\mathcal{H(V,E)}$
    \begin{equation*}
        D_v=\mathrm{diag}( A, \underbrace{B}_{(n-1)- times}, \underbrace{C}_{n-times}, D, \underbrace{E}_{(n-1)-times}, \underbrace{F}_{n-times} ),
    \end{equation*} when $n$ is odd and $A, B, C, D, E$ and $F$ are given in Theorem \ref{degree theorem odd}.\\

  \noindent From the Definition \ref{De matrix}, the degree-edge matrix of the associating hypergraph  $\mathcal{H(V,E)}$ is $D_e=\mathrm{diag}( \underbrace{3, 3, 3, \ldots, 3}_{\mid \mathcal{E}\mid})$, where $\mathcal{E}$ is the size of the hypergraph $\mathcal{H(V,E)}$.

\begin{theorem}
    The independence number $\alpha(\mathcal{H})$ of the associating hypergraph $\mathcal{H(V,E)}$ is $ \frac{n}{2}+4,$ when $n$ is even and $n+2,$ when $n$ is odd.
\end{theorem}

\begin{proof}
   Suppose $n$ is even. Since $\frac{n}{2}$ vertices of type $(g,1),$ where $ g\in S$ does not form hyperedges with itself. Therefore, we have $\frac{n}{2}$ vertices that are independent from each other. We can not choose more than two vertices from the type $(g,1),$ where $ g\in R$ otherwise they form hyperedges. We have to choose these two vertices such that they do not form hyperedges with $\frac{n}{2}$ elements of type $(g,1), g\in S$. Also, we can not choose more than two vertices from $(g,0),$ where $g\in S$, otherwise they also form hyperedges. Choose these two vertices of type $(g,0),$ where $g\in S$ such that they do not form hyperedges with set of $\frac{n}{2}+2$ vertices we have selected above. Hence the maximum independent set consists of $\frac{n}{2}+2+2$ elements.
    Therefore, the independence number $\alpha(\mathcal{H})= \frac{n}{4}+4.$\\

    \noindent Now, let $n$ be odd. The proof is similar as above. One can check that  $n$ vertices of type $(g,1), g \in S$ and two vertices of type $(g,0), g\in S$ make the maximum independent set. Hence, the independence number $\alpha(\mathcal{H})= n+2.$
\end{proof}

\begin{theorem}
    The transversal number $\tau(\mathcal{H})$ of the associating hypergraph $\mathcal{H(V,E)}$  is $\frac{7n}{2}-4,$ when $n$ is even and $3n-2,$ when $n$ is odd.
\end{theorem}

\begin{proof}
    We can compute the transversal number by selecting vertices that can cover all the edges, but there is no empty edge is $\mathcal{H(V,E)}$, the property given in Lemma \ref{trans property} hold.\\
    Suppose $n$ is even. Then,
    \begin{align*}
        \alpha(\mathcal{H})+ \tau(\mathcal{H})=& 4n\\
   \frac{n}{2}+4 + \tau(\mathcal{H})=& 4n\\
    \tau(\mathcal{H})=& \frac{7n}{2}-4.
    \end{align*}

   \noindent  Now, let $n$ is odd. The proof is similar as above. Hence, the transversal number $\tau(\mathcal{H})$ of the associating hypergraph $\mathcal{H(V,E)}$ is $3n-2.$
\end{proof}

\begin{theorem}\label{covering number}
    The covering number $\rho(\mathcal{H})$ of the associating hypergraph $\mathcal{H(V,E)}$ is $n+ \lceil \frac{n}{3} \rceil.$
\end{theorem}

\begin{proof}
   Suppose $n$ is even. Since there are $3n$ elements of type $(g,0)$ and $\{(g,1),$ where $g\in R\},$ they always form $n$ disjoint hyperedges with itself. Therefore, with those $n$ disjoint hyperedges, we can cover $3n$ vertices of $M(D_n,2).$ Now, we have left $n$ elements of form $(g,1),$ where $ g \in S,$ they can form $\lfloor \frac{n}{3} \rfloor$ disjoint hyperedges with itself. We cover $3 \times \lfloor\frac{n}{3} \rfloor$ vertices with these $\lfloor \frac{n}{3} \rfloor$ disjoint hyperedges. If $n$ is not divisible by $3,$ then one or two vertices are left. Then we need one more hyperedge to cover these vertices. So, we can cover all $n$ vertices of type $(g,1), g\in S$ by $\lceil \frac{n}{3}\rceil$ hyperedges. Therefore, the covering number $\rho(\mathcal{H})$ of the associating hypergraph $\mathcal{H(V,E)}$ is $n+ \lceil \frac{n}{3} \rceil.$ \\

   \noindent Now, let $n$ is odd. Here, elements of type $\{(g,0),$ where $g \in S\}$ and $(g,1)$ will form $n$ disjoint hyperedges. Elements of type $(g,0),$ where $ g \in R$ will make $\lfloor \frac{n}{3} \rfloor$ disjoint hyperedges. If $n$ is not divisible by $3$, then remaining one or two vertices of $(g,0),$ where $ g \in R$ can be cover by one more hyperedge. Hence, the covering number $\rho(\mathcal{H})$ of the associating hypergraph $\mathcal{H(V,E)}$ is also $ n+ \lceil \frac{n}{3} \rceil.$ 
\end{proof} 

    \begin{theorem}
        The chromatic number $\psi(\mathcal{H})$ of the associating hypergraph $\mathcal{H(V,E)}$ is $\frac{7n}{4}-1,$ when $n$ is even and ${3n}-{1},$ when $n$ is odd.
    \end{theorem}

\begin{proof}
    Suppose $n$ is even. We have $\frac{n}{2}+4$ maximum independence set of vertices. If we color them by one color, then we have left $\frac{7n}{2}-4$ vertices. These  $\frac{7n}{2}-4$ vertices form hyperedges with each other but if we color every two vertices with one different color, then in associating hypergraph $\mathcal{H(V,E)}$ there is no hyperedge with all three vertices has the same color. Therefore, the chromatic number $\psi(\mathcal{H})= (\frac{7n}{4}-2)+1= \frac{7n}{4}-1.$\\

    \noindent Now, let $n$ is odd. The proof is similar to the above. Hence, the chromatic number $\psi(\mathcal{H})= ({3n}-2)+1= 3n-1.$
\end{proof}

\begin{remark}
    \cite[p.1]{hardness} The associating hypergraph $\mathcal{H(V,E)}$ with $\frac{7n}{4}-1$-colorable and $3$-uniform hypergraph with $c$-colors is $np$-hard where $c > \frac{7n}{4}-1,$ when $n$ is even. Similarly, the associating hypergraph $\mathcal{H(V,E)}$ with $3n-1$-colorable and $3$-uniform hypergraph with $c$-colors is $np$-hard where $c > 3n-1,$ when $n$ is odd.
\end{remark}

 \begin{theorem}
     The strong chromatic number $\overline{\psi}(\mathcal{H})$ of the associating hypergraph $\mathcal{H(V,E)}$ is ${4n}.$
 \end{theorem}

\begin{proof}
    Suppose $n$ is even. We have $4n$ different non-isolated vertices in which any two vertices is a part of some hyperedges, so we have to color this $4n$ vertices with $4n$ different colors. Therefore,  $\overline{\psi}(\mathcal{H})=4n$

    \noindent Now, let $n$ is odd. Proof is similar to the above. Hence the strong chromatic number $\overline{\psi}(\mathcal{H})$ is also ${4n}.$
\end{proof}

 \begin{theorem}\label{matching number}      The matching number $\upsilon(\mathcal{H})$ of the associating hypergraph $\mathcal{H(V,E)}$   is $ n+ \lfloor \frac{n}{3} \rfloor.$     \end{theorem}

\begin{proof}  Suppose $n$ is even. Since in the proof of Theorem \ref{covering number}, we can see that there is $ n+ \lfloor \frac{n}{3} \rfloor$ hyperedges that are disjoint. Therefore, matching number $\upsilon(\mathcal{H})$ of the associating hypergraph $\mathcal{H(V,E)}$  is $ n+ \lfloor \frac{n}{3} \rfloor.$  \\
\noindent Now, let $n$ is odd. Proof is similar to the above. Hence, the matching number $\upsilon(\mathcal{H})$ is also $ n+ \lfloor \frac{n}{3} \rfloor.$\end{proof}

\begin{theorem} \label{matching theorem}
    The matching polynomial $\textit{M}(\mathcal{H},w)$ of the associating hypergraph $\mathcal{H(V,E)}$ is $\displaystyle\sum_{k=0}^ {n+ \lfloor \frac{n}{3} \rfloor} a_k \textit{w}_1^{4n-3k} \textit{w}_2 ^{k}$.
\end{theorem}

 \begin{proof}
     Since $\mathcal{H(V,E)}$ is $3$-uniform hypergraph. Let $\textit{M}$ be a $k$-matching in $\mathcal{H(V,E)}$  and let us assign weights $\textit{w}_1$ and $\textit{w}_2$ to each vertices and hyperedges in $\textit{M}$ respectively. If $a_k$ is the number of $k$-matching in $\mathcal{H(V,E)}$ and from Theorem \ref{matching number}, the matching number $\upsilon(\mathcal{H})$ is the largest $k$ for which $a_k$ is non-zero is $ n+ \lfloor \frac{n}{3} \rfloor$.  Then by \cite[p.2]{matching polynomial}, the matching polynomial $\textit{M}(\mathcal{H},w)$ of the associating hypergraph $\mathcal{H(V,E)}$ is  $\displaystyle\sum_{k=0}^ {n+ \lfloor \frac{n}{3} \rfloor} a_k \textit{w}_1^{4n-3k} \textit{w}_2 ^{k}.$
 \end{proof}

\noindent \textbf{Future Scope:}  One may be intersted in classifying the loops by their associating hypergraphs. \\ By \cite[Theorem 4.4, p.774]{diassociativity}, one can observe that if three elements $x, y$ and $z$ in a conjugacy closed loop associate in an order, then they associate in any order. This implies that the associating hypergraph for conjugacy closed loop is $3$-uniform undirected hypergraph. One may be intersted to explore hypergraph theoretic properties for such loops. \\ One may also be intersted to answer the following question:\\ Given a $3$-uniform hypergraph $H,$ does there exist a loop $L$ whose associating hypergraph is $H?$\\


\noindent \textbf{Acknowledgement:} The first author is supported by junior research fellowship of CSIR, India.


\normalsize \begin{thebibliography}{99}
\bibitem[1]{berge} Claude Berge. Graphs and Hypergraphs. North-Holland Publishing Company, Amsterdam, 1973.

\bibitem[2]{bela} Bollobás, Béla. Extremal Graph Theory. Academic Press, 1978.

\bibitem[3]{hamiltonian chain} Gyula Y. Katona and Hal A. Kierstead. Hamiltonian chains in hypergraphs. Journal of Graph Theory 30.3 (1999): 205-212.

\bibitem[4]{loose hamilton} Daniela Kühn and Deryk Osthus. Loose Hamilton cycles in 3-uniform hypergraphs of high minimum degree. Journal of Combinatorial Theory, Series B 96.6 (2006): 767-821.

\bibitem[5]{chromatic polynomial} Mieczysław Borowiecki and Ewa Łazuka. Chromatic polynomials of hypergraphs. Discussiones Mathematicae Graph Theory 20.2 (2000): 293-301.

\bibitem[6]{hardness} Irit Dinur, Oded Regev and Clifford Smyth. The hardness of 3-uniform hypergraph coloring. Combinatorica 25.5 (2005): 519-535.

\bibitem[7]{partioning} Jie Ma and Xingxing Yu. Partitioning 3-uniform hypergraphs. Journal of Combinatorial Theory, Series B 102.1 (2012): 212-232.

\bibitem[8]{eulerian} M. Amin Bahmanian and Mateja Šajna. Eulerian properties of hypergraphs. arXiv preprint arXiv:1608.01040 (2016).

\bibitem[9]{linear cycle} Beka Ergemlidze, Ervin Gyori and Abhishek Methuku. 3-uniform hypergraphs and linear cycles. SIAM Journal on Discrete Mathematics 32.2 (2018): 933-950.

\bibitem[10]{matching polynomial} Zhiwei Guo, Haixing Zhao and Yaping Mao. On the matching polynomial of hypergraphs. Journal of Algebra Combinatorics Discrete Structures and Applications 4.1 (2017): 1-11.

\bibitem[11]{spectral properties} Jiang Zhou et al. Some spectral properties of uniform hypergraphs. arXiv preprint arXiv:1407.5193 (2014).

\bibitem[12]{prime} Abderrahim Boussaïri et al. Prime 3-uniform hypergraphs. Graphs and Combinatorics 37.6 (2021): 2737-2760.

\bibitem[13]{genus} Yifan Jing and Bojan Mohar. The genus of complete 3-uniform hypergraphs. Journal of Combinatorial Theory, Series B 141 (2020): 223-239.

\bibitem[14]{groups} Peter J Cameron, Aparna Lakshmanan S and Midhuna V Ajith. Hypergraphs defined on algebraic structures. Communications in Combinatorics and Optimization, 2024, doi: 10.22049/cco.2024.29607.2077

\bibitem[15]{chemistry} Elena V. Konstantinova and Vladimir A. Skorobogatov. Application of hypergraph theory in chemistry. Discrete Mathematics 235.1-3 (2001): 365-383.

\bibitem[16]{biology} Steffen Klamt, Utz-Uwe Haus and Fabian Theis. Hypergraphs and cellular networks. PLoS computational biology 5.5 (2009): e1000385.

\bibitem[17]{spectral hypergraph} YueGao et al. Visual-textual joint relevance learning for tag-based social image search. IEEE Transactions on Image Processing 22.1 (2012): 363-376.

\bibitem[18]{data mining} Céline Hébert, Alain Bretto and Bruno Crémilleux. A data mining formalization to improve hypergraph minimal transversal computation. Fundamenta Informaticae 80.4 (2007): 415-433.

\bibitem[19]{image} Alain Bretto and Luc Gillibert. Hypergraph-based image representation. Graph-Based Representations in Pattern Recognition: 5th IAPR International Workshop, GbRPR 2005, Poitiers, France, April 11-13, 2005. Proceedings 5. Springer Berlin Heidelberg, 2005.

\bibitem[20]{machine learning} Jin Huang, Rui Zhang and Jeffrey Xu Yu. Scalable hypergraph learning and processing. 2015 IEEE International Conference on Data Mining. IEEE, 2015.

\bibitem[21]{machine}  Samuel Bulò and Marcello Pelillo. A game-theoretic approach to hypergraph clustering. Advances in neural information processing systems 22 (2009).

\bibitem[22]{geometry} Shakhar Smorodinsky. On the chromatic number of geometric hypergraphs. SIAM Journal on Discrete Mathematics 21.3 (2007): 676-687.

\bibitem[23]{bretto} Alain Bretto. Hypergraph theory. An introduction. Mathematical Engineering. Cham: Springer 1 (2013).

\bibitem[24]{formula} Michael A. Henning and Anders Yeo. Transversals in linear uniform hypergraphs. Cham: Springer International Publishing, 2020.

\bibitem[25]{Dv De} Marianna Bolla. Spectra, euclidean representations and clusterings of hypergraphs. Discrete Mathematics 117.1-3 (1993): 19-39.

\bibitem[26]{JAR}Juan Alberto Rodriguez. On the Laplacian spectrum and walk-regular hypergraphs. Linear and Multilinear Algebra 51.3 (2003): 285-297.

\bibitem[27]{quasigroup book} Victor Shcherbacov. Elements of quasigroup theory and applications. Chapman and Hall/CRC, 2017.

\bibitem[28]{Moufang loop} Andrew Rajah. An alternative construction of the Moufang loop M (G, 2). Proceedings of ICMSA 2.1 (2006).

\bibitem[29]{diassociativity} Michael K. Kinyon, Kenneth Kunen and J. D. Phillips. Diassociativity in conjugacy closed loops. Communications in Algebra 32.2 (2004): 767-786.


\end{thebibliography}
\end{document}